\documentclass[leqno,11pt]{amsart}
\usepackage{amssymb,amsthm,wasysym}
\usepackage[latin2]{inputenc}

\theoremstyle{plain}

\allowdisplaybreaks

\newtheorem{thm}{Theorem}[section]
\newtheorem{cor}[thm]{Corollary}
\newtheorem{pro}[thm]{Proposition}

\newtheorem{lem}[thm]{Lemma}

\newtheorem{defin}[thm]{Definition}
\newtheorem*{ex*}{Example}

\numberwithin{equation}{section}

\newcommand{\Z}{\mathbb{Z}_{2}^{d}}
\newcommand{\N}{\mathbb{N}}
\newcommand{\R}{\mathbb{R}_{+}^{d}}
\newcommand{\RR}{\mathbb{R}^{d}}
\newcommand{\La}{\mathcal{L}_{\alpha}}
\def\a{\alpha}
\def\z{\zeta}
\def\ll{\lambda}

\def\eps{\varepsilon}
\def\b{(\z,q_{\pm})}

\DeclareMathOperator{\e}{{\mathbb{E}\textrm{xp}}}
\DeclareMathOperator{\ee}{{\mathbb{E}\textrm{\emph{xp}}}}

\DeclareMathOperator{\domain}{Dom}

\DeclareMathOperator{\spec}{spec}

\begin{document}

\subjclass[2000]{42C10 (primary), 42B15, 42B20 (secondary)}
\keywords{multiplier, Dunkl harmonic oscillator, generalized Hermite expansions, Laguerre expansions of convolution type, Calder\'on-Zygmund operator, $A_p$ weight}


\title[Multipliers in Dunkl and Laguerre settings]{Multipliers of Laplace transform type in certain Dunkl and Laguerre settings}
\author[]{Tomasz Szarek}
\address{Tomasz Szarek,     \newline
      ul{.} W{.} Rutkiewicz 29\slash 43,
      PL-50--571 Wroc\l{}aw, Poland        \vspace{10pt}}

\email{szarektomaszz@gmail.com}

\begin {abstract}
We investigate Laplace type and Laplace-Stieltjes type multipliers in the $d$-dimen\-sional setting of the Dunkl harmonic oscillator with the associated group of reflections isomorphic to $\Z$ and in the related context of Laguerre function expansions of convolution type. We use Calder\'on-Zygmund theory to prove that these multiplier operators are bounded on weighted $L^p$, $1<p<\infty$, and from $L^1$ to weak $L^1$.
\end {abstract}

\maketitle

\section{Introduction}\label{intro}
\setcounter{equation}{0}

In \cite{Sz1} the author defined and investigated square functions related to the Dunkl harmonic oscillator $L_{\a}$ and the related group of reflections isomorphic to $\Z$. This paper continues the study of $L^p$ mapping properties of fundamental harmonic analysis operators associated with $L_{\a}$. We consider Laplace type and Laplace-Stieltjes type multiplier operators, see Section \ref{pre} for the definitions. The most typical examples of such operators are the imaginary powers $L_{\a}^{-i\gamma}$, $\gamma\in\mathbb{R}$, and the fractional integral operators $L_{\a}^{-\delta}$, $\delta>0$, as pointed out in \cite[Section 2]{W}. Thus our main result, Theorem \ref{main} below, may be regarded as a continuation of the investigations of Nowak and Stempak \cite{NS2,NS4}, where these operators were studied in the context of $L_{\a}$. A trivial choice of the multiplicity function reduces the Dunkl setting to the situation of classical Hermite function expansions. Thus, in particular, our considerations provide results in the Hermite setting, where the Laplace type multiplier operators were implicitly analyzed by Stempak and Torrea, see the comment in \cite[p.\,46]{StTo2}. Moreover, our Dunkl situation reduces to the setting of Laguerre function expansions of convolution type after restricting to reflection invariant functions. Consequently, we obtain also results in the Laguerre context. 

Multipliers related to numerous classic kinds of orthogonal expansions were widely investigated. In particular, Stempak and Trebels \cite{StTr} studied multipliers of non-Laplace type in a one-dimensional Laguerre setting, which is deeply connected with our Dunkl situation (see Remark \ref{Lag} below). Some earlier results concerning multiplier operators of Laplace type for discrete and continuous orthogonal expansions can be found in \cite{BCN,BMR,DDD,M,Sa2,W}, among others.
A general treatment of Laplace type multipliers in a context of symmetric diffusion semigroups can be found in Stein's monograph \cite{Ste}.

We refer the reader to the survey article by R\"osler \cite{R} for basic facts concerning Dunkl's theory. A precise description of the Dunkl framework for the particular group of reflections $G$ isomorphic to $\Z$ can be found for instance in \cite[Section 3]{NS3}. Here we invoke only the most relevant facts. We shall work on the space $\RR$, $d\geq 1$, equipped with the measure
\begin{align*}
dw_{\a}(x)=\prod_{j=1}^{d}|x_{j}|^{2\a_{j}+1}\,dx,\qquad x=(x_1,\ldots,x_d)\in\RR,
\end{align*}
and with the Euclidean norm $|\cdot|$. The multi-index $\a=(\a_1,\ldots,\a_d)\in [-1\slash 2,\infty)^d$ represents the multiplicity function. Consider the reflection group $G\simeq\Z$ generated by $\sigma_{j}$, $j=1,\ldots,d$,
$$
\sigma_{j}(x_{1},\ldots,x_{j},\ldots,x_{d})=(x_{1},\ldots,-x_{j},\ldots,x_{d}).
$$
Clearly, the reflection $\sigma_{j}$ is in the hyperplane orthogonal to $e_j$, the $j$th coordinate vector. Notice that the measure $w_{\a}$ is $G$-invariant. The Dunkl differential-difference operators $T_{j}^{\a}$, $j=1,\ldots,d$, are given by
\begin{align*}
T_{j}^{\a}f(x)=\partial_{x_{j}}f(x)+(\a_{j}+1\slash 2)\frac{f(x)-f(\sigma_{j}x)}{x_{j}},\qquad f\in C^{1}(\mathbb{R}^{d}),\qquad j=1,\ldots,d.
\end{align*}
The Dunkl Laplacian,
\begin{align*}
\Delta_{\a}f(x)=\sum_{j=1}^{d}\big(T_{j}^{\a}\big)^{2}f(x)
=
\sum_{j=1}^{d}\bigg(\frac{\partial^{2}f}{\partial x_j^2}(x)+\frac{2\a_j+1}{x_j}\frac{\partial f}{\partial x_j}(x)-(\a_j+1\slash 2)\frac{f(x)-f(\sigma_j x)}{x_j^2}\bigg),
\end{align*} 
is formally self-adjoint in $L^2(\RR,dw_{\a})$. The Dunkl harmonic oscillator is defined as
\begin{align*}
L_{\a}=-\Delta_{\a}+|x|^{2}.
\end{align*}
This operator will play in our investigations a similar role to that of the Euclidean Laplacian in the classical harmonic analysis. Note that for $\a=(-1\slash 2,\ldots,-1\slash 2)$, $L_\a$ becomes the classic harmonic oscillator $-\Delta+|x|^2$. We shall consider a self-adjoint extension $\mathcal{L}_\a$ of $L_\a$, whose spectral decomposition is discrete and given by the generalized Hermite functions $h_n^\a$, see Section \ref{pre} for details. 

The main objects of our study are spectral multipliers associated with $\La$. More precisely, we investigate two kinds of such operators, see Definition \ref{def} below. The first one is a multiplier of Laplace type, which originates in Stein's monograph \cite{Ste}. The second one is a multiplier of Laplace-Stieltjes type, which was considered recently by Wr\'obel \cite{W} in the context of Laguerre function expansions of Hermite type, and its definition has roots in the work of De N\'apoli, Drelichman and Dur\'an \cite{DDD}.
Our main result, Theorem \ref{main}, says that the multiplier operators in question are bounded on weighted $L^p(dw_\a)$, $1<p<\infty$, and satisfy weighted weak type (1,1) inequality, for a large class of weights.
We note that the unweighted $L^p$-boundedness, $1<p<\infty$, of the Laplace type multiplier operators follows also from the refinement of Stein's general Littlewood-Paley theory for semigroups (see \cite[Corollary 3, p.\,121]{Ste}) due to Coifman, Rochberg and Weiss \cite{CRW}, see also \cite[Theorem 2]{Me}.

In the proof of Theorem \ref{main} we exploit the arguments from \cite{NS2} that allow us to reduce the analysis to suitably defined Laguerre-type operators related to the smaller measure space $(\R,dw_{\a}^{+})$; here $\R=(0,\infty)^d$ and $w_\a^+$ is the restriction of $w_\a$ to $\R$. Then we apply the theory of Calder\'on-Zygmund operators with the underlying space of homogeneous type $(\R,dw_{\a}^{+},|\cdot|)$. An essential technical difficulty connected with this approach is to show the relevant kernel estimates. 
Here we employ a convenient technique having roots in Sasso's paper \cite{Sa1} and developed by Nowak and Stempak \cite{NS1,NS3,NS2} and the author \cite{Sz,Sz1}. It is remarkable that similar methods have been established recently in the contexts of Jacobi expansions \cite{NoSj} and Bessel operators \cite{BCN}.

The paper is organized as follows. Section \ref{pre} contains the setup, definitions of the investigated multipliers and statements of the main results. Also, suitable Laguerre-type operators, related to the restricted space $(\R,dw_{\a}^{+})$, are defined and the proof of the main theorem is reduced to showing that these auxiliary operators are Calder\'on-Zygmund operators. Furthermore, we verify that the Laguerre-type operators are associated, in the Calder\'on-Zygmund theory sense, with suitable integral kernels. The section is concluded by comments on the relation between the Leguerre-type operators and the Laguerre setting studied in \cite{NS1}.    
Finally, Section \ref{ker} is devoted to the proofs of the relevant kernel estimates.

Throughout the paper we use a standard notation with essentially all symbols referring to the spaces $(\RR,dw_{\a},|\cdot|)$ or $(\R,dw_{\a}^{+},|\cdot|)$. Thus $\Delta$ and $\nabla$ denote the Euclidean Laplacian and gradient, respectively. Further, $L^{p}(\RR,Wdw_{\a})$ stands for the weighted $L^{p}(\RR,dw_{\a})$ space, $W$ being a nonnegative weight on $\RR$; we write simply $L^{p}(dw_{\a})$ if $W\equiv 1$. 
By $\langle f,g\rangle_{dw_{\a}}$ we mean $\int_{\RR}f(x)\overline{g(x)}\,dw_{\a}(x)$ whenever the integral makes sense. In a similar way we define $L^{p}(\R,Udw_{\a}^{+})$ and $\langle f,g\rangle_{dw_{\a}^+}$. For  $1\leq p<\infty$ we denote by $A_{p}^{\a,+}$ the Muckenhoupt class of $A_{p}$ weights associated to the space $(\R,dw_{\a}^+,|\cdot|)$. 

While writing estimates we will frequently use the notation $X\lesssim Y$ to indicate that $X\leq CY$ with a positive constant $C$ independent of significant quantities. We will write $X\simeq Y$ when both $X\lesssim Y$ and $Y\lesssim X$.

\textbf{Acknowledgments.} The author would like to thank Professor Krzysztof Stempak for suggesting the topic, Dr.\ Adam Nowak for valuable comments during preparation of the paper, and B\l a\.zej Wr\'obel for sharing the manuscript \cite{W} and discussions concerning the topic.

\section{Preliminaries and main results}\label{pre}
\setcounter{equation}{0}

Let $k=(k_1,\ldots,k_d) \in \N^d$, $\N=\{0,1,\dots\}$,  and
$\a = (\a _1, \ldots , \a _d) \in [-1\slash 2,\infty)^{d}$
be multi-indices. The generalized Hermite functions in $\RR$ are defined as tensor products
$$
h_{k}^{\alpha}(x) = h _{k_1}^{\alpha _1}(x_1) \cdot \ldots \cdot
h_{k_d}^{\alpha _d}(x_d), \qquad x = (x_1, \ldots ,x_d)\in \RR,
$$
where $h_{k_i}^{\alpha _i}$ are the one-dimensional generalized Hermite functions
\begin{align*}
h_{2k_i}^{\alpha_i}(x_i)&=d_{2k_i,\alpha_i}\, e^{-x_i^2/2}L_{k_i}^{\alpha_i}(x_i^2),\\
h_{2k_i+1}^{\alpha _i}(x_i)&=d_{2k_i+1,\alpha_i}\, e^{-x_i^2/2}x_iL_{k_i}^{\alpha_i+1}(x_i^2);
\end{align*}
here $L_{k_i}^{\alpha_i}$ is the Laguerre polynomial of degree $k_{i}$ and order $\a_i$, and $d_{n,\a_{i}}$, $n\in\N$, are suitable normalizing constants, see \cite[p.\,544]{NS3} or \cite[p.\,4]{NS2}. The system $\{h_{k}^{\a} : k\in\N^{d} \}$ is an orthonormal basis in $L^{2}(\RR,dw_{\a})$ consisting of eigenfunctions of $L_{\a}$,
\begin{align*}
L_{\a}h_{k}^{\a}=\ll_{|k|}^{\a} h_{k}^{\a},\qquad \ll_n^\a=2n+2|\a|+2d,\qquad n\in\N;
\end{align*}
here $|k|=k_1+\ldots+k_d$ is the length of $k$; similarly $|\a|=\a_1 + \ldots + \a_d$. The operator
\begin{align*}
\mathcal{L}_\alpha f=\sum_{n=0}^\infty\ll_{n}^{\a}
\sum_{|k|=n}\langle f,h_k^{\alpha}
  \rangle_{dw_{\alpha}}h_k^{\alpha},
\end{align*}
defined on the domain $\domain(\mathcal{L}_{\a})$ consisting of all functions $f\in L^2(\RR,dw_{\a})$ for which the defining series converges in $L^2(\RR,dw_{\a})$, is a self-adjoint extension of $L_\a$ considered on $C_c^\infty(\RR)$.

The heat semigroup associated with $\mathcal{L}_{\a}$, defined by the spectral theorem as
$$
T_t^\a f = \exp(-t\mathcal{L}_{\a})f
=
\sum_{n=0}^{\infty} e^{-t\ll_{n}^{\a}}\sum_{|k|=n}\langle f,h_k^{\alpha}
  \rangle_{dw_{\alpha}}h_k^{\alpha}, 
\qquad f\in L^2(\RR,dw_{\alpha}),
$$  
is a strongly continuous semigroup of contractions on $L^2(\RR,dw_{\a})$. We have the integral representation
\begin{equation*}
  T_t^\alpha f(x)=\int_{\RR} G_t^\alpha(x,y)f(y)\,dw_{\alpha}(y),
  \qquad x\in \RR, \quad t>0,
\end{equation*}
where the Dunkl heat kernel is given by
\begin{equation}\label{gie}
G^\alpha_t(x,y)=\sum_{n=0}^\infty e^{-t\ll_{n}^{\a}} \sum_{|k|=n}
h_k^\alpha(x)h_k^\alpha(y).
\end{equation}
This oscillating series can be summed, see for instance \cite[p.\,544]{NS3} or \cite[p.\,5]{NS2}, and the resulting formula is
\begin{align*}
G^{\a}_{t}(x,y)=\sum_{\eps\in\Z}G^{\a,\eps}_{t}(x,y),
\end{align*}
with the component kernels
$$
G^{\a,\eps}_{t}(x,y)=(2\sinh 2t)^{-d}\exp\Big({-\frac{1}{2} \coth(2t)\big(|x|^{2}+|y|^{2}\big)}\Big)
\prod^{d}_{i=1} (x_i y_i)^{\eps_{i}} \frac{I_{\alpha_i+\eps_{i}}\left(\frac{x_i y_i}{\sinh 2t}\right)}{(x_i y_i)^{\a_{i}+\eps_{i}}},
$$
where $I_{\nu}$ denotes the modified Bessel function of the first kind and order $\nu$. Here we consider the functions $z\mapsto z^{\nu}$ and the Bessel function as analytic functions on $\mathbb{C}$ cut along the half-axis $\{ix : x\leq 0\}$, see the references given above. Note that each $G_{t}^{\a,\eps}(x,y)$, $\eps \in \Z$, is also expressed by the series \eqref{gie}, but with the summation in $k$ restricted to the set
$$
\mathcal{N}_{\eps} = \big\{ k \in \N^{d} : k_i \; \textrm{is even if} \; \eps_i=0,
	\; k_i \; \textrm{is odd if}\; \eps_i=1,\; i=1,\ldots,d\big\}.
$$

\begin{defin}\label{def}
Following E{.} M{.} Stein \cite[p.\,58,\,p.\,121]{Ste} we say that $m$ is a multiplier of Laplace (transform) type associated with $\La$ if $m$ has the form
\begin{align}\label{molt}
m(z)=m_{\eta}(z)=z\int_0^\infty e^{-tz}\eta(t)\,dt, \qquad z \geq \ll_0^{\a},
\end{align} 
where $\eta$ is a bounded measurable function on $(0,\infty)$. We also consider multipliers of Laplace-Stieltjes type associated with $\La$ (see \cite[Section 2]{W} and comments therein) having the form
\begin{align}\label{momt}
m(z)=m_{\mu}(z)=\int_0^\infty e^{-tz}\,d\mu(t), \qquad z \geq \ll_0^{\a},
\end{align} 
where $\mu$ is a signed or complex Borel measure on $(0,\infty)$ with total variation $|\mu|$ satisfying
\begin{align}\label{meas}
\int_0^\infty e^{-t\ll_0^{\a}}\,d|\mu|(t)<\infty.
\end{align} 
\end{defin}

We are now prepared to define the main objects of our study. Given a Laplace type or Laplace-Stieltjes type multiplier $m$, we consider the multiplier operators
\begin{align}\label{mla}
m(\La) f=
\sum_{n=0}^{\infty} m(\ll_n^\a)\sum_{|k|=n}\langle f,h_k^{\alpha}
  \rangle_{dw_{\alpha}}h_k^{\alpha}, 
\qquad f\in L^2(\RR,dw_{\alpha}).
\end{align} 
It is not hard to see that $m$ defined either by \eqref{molt} or by \eqref{momt} is a bounded function on the interval $[\ll_{0}^{\a},\infty)$. Since $\spec(\La)\subset [\ll_{0}^{\a},\infty)$, the operator $m(\La)$ is a well-defined bounded operator on $L^2(\RR,dw_\a)$.

To state and prove our main result it is convenient to introduce the following terminology. Given $\eps\in\Z$, we say that a function $f\colon\RR\to\mathbb{C}$ is $\eps$-symmetric if for each $j=1,\ldots,d$, $f$ is either even or odd with respect to the $j$th coordinate according to whether $\eps_{j}=0$ or $\eps_{j}=1$, respectively. If $f$ is $(0,\ldots,0)$-symmetric, then we simply say that $f$ is symmetric. Further, if there exists $\eps\in\Z$ such that $f$ is $\eps$-symmetric, then we denote by $f^{+}$ the restriction of $f$ to $\R$. This convention pertains also to $\eps$-symmetric weights defined on $\RR$.

The main result of the paper reads as follows.

\begin{thm}\label{main}
Assume that $\a\in[-1\slash 2,\infty)^{d}$, $W$ is a weight on $\RR$ invariant under the reflections $\sigma_1, \ldots,\sigma_d$, and $m$ is as in \eqref{molt} or \eqref{momt}. Then the multiplier operator $m(\La)$, defined initially on $L^2(\RR,dw_\a)$ by \eqref{mla}, extends uniquely to a bounded linear operator on $L^{p}(\RR,W dw_{\a})$, $W^+\in A_{p}^{\a,+}$, $1<p<\infty$, and to a bounded linear operator from $L^1(\RR,W dw_{\a})$ to weak $L^1(\RR,W dw_{\a})$, $W^+\in A_{1}^{\a,+}$.
\end{thm} 

The proof we give relies on reducing the problem to showing similar mapping properties for certain operators emerging from $m(\La)$ and related to the restricted space $(\R,dw_\a^+)$. The details are as follows. 
Let $1 \le p < \infty$ be fixed and let $W$ be a symmetric weight on $\RR$ such that $W^+ \in A_p^{\a,+}$.
We decompose $m(\La)$ into a finite sum of $L^2$-bounded operators
\begin{align*}
m(\La)=\sum_{\eps\in\Z} m^{\eps}(\mathcal{L}_{\a}),
\end{align*} 
where
$$
m^{\eps}(\mathcal{L}_{\a})
=
\sum_{n=0}^{\infty} m(\ll_n^\a)
\sum_{\substack{|k|=n\\ k\in \mathcal{N}_{\eps}}}
\langle f,h_k^{\alpha}
  \rangle_{dw_{\alpha}}h_k^{\alpha}, 
\qquad f\in L^2(\RR,dw_{\alpha}).
$$
Then we proceed as in \cite[Section 3]{NS2}, see also \cite[Section 2]{Sz1}. We split a function $f\in L^2(\RR,dw_\a)$ into a finite sum of $\eps$-symmetric functions $f_\eps$,
\begin{align*}
f=\sum_{\eps\in\Z}f_\eps.
\end{align*} 
We also introduce the Laguerre-type operators acting on $L^2(\R,dw_\a^+)$,
\begin{align}\label{eq2}
m^{\eps,+}(\mathcal{L}_{\a}) f
=
\sum_{n=0}^{\infty} m(\ll_n^\a)
\sum_{\substack{|k|=n\\ k\in \mathcal{N}_{\eps}}}
\langle f,h_k^{\alpha}
  \rangle_{dw_{\alpha}^+}h_k^{\alpha}.
\end{align}
Since $h_k^\a$ is $\eps$-symmetric if and only if $k\in \mathcal{N}_\eps$, we see that
\begin{align*}
m(\La) f=
\sum_{\eps\in\Z} m^{\eps}(\La) f_{\eps}, \qquad f\in L^2(\RR,dw_\a).
\end{align*}
Taking into account the fact that $m^{\eps}(\La)f_\eps$ is $\eps$-symmetric, and the relation $\langle f_{\eps}, h_{k}^{\a}\rangle_{dw_{\a}}=2^{d}\langle f_{\eps}^{+}, h_{k}^{\a}\rangle_{dw_{\a}^{+}}$ for $k\in\mathcal{N}_{\eps}$, we obtain
\begin{align*}
\|m(\La) f\|_{L^p(\RR,Wdw_\a)}
\leq&\,
2^{d\slash p} \sum_{\eps\in\Z} \|m^{\eps}(\La)f_{\eps}\|_{L^p(\R,W^+dw_\a^+)}\\
=&\,
2^{d+d\slash p}\sum_{\eps\in\Z} \|m^{\eps,+}(\La)f_{\eps}^+\|_{L^p(\R,W^+dw_\a^+)},
\end{align*}
for $f\in L^2(\RR,dw_\a)\cap L^p(\RR,Wdw_\a)$. Since we have, see \cite[p.\,6]{NS2} for the unweighted case,
\begin{align*}
\|f\|_{L^p(\RR,Wdw_{\alpha})} \simeq \sum_{\varepsilon \in \mathbb{Z}_2^d} 
\|f^{+}_{\varepsilon}\|_{L^p(\mathbb{R}^d_+,W^{+}dw_{\alpha}^+)},
\end{align*}
the above estimates together with the bounds
\begin{align*}
\|m^{\eps,+}(\La) f_{\eps}^+\|_{L^p(\R,W^+dw_\a^+)}
\lesssim
\|f_{\eps}^+\|_{L^p(\R,W^+dw_\a^+)},\qquad \eps\in\Z,
\end{align*}
imply the estimate
\begin{align*}
\|m(\La) f\|_{L^p(\RR,Wdw_\a)}
\lesssim
\|f\|_{L^p(\RR,Wdw_\a)}.
\end{align*}
An analogous implication involving weighted weak type $(1,1)$ inequalities is also valid. Thus we reduced proving Theorem \ref{main} to showing the following.

\begin{thm}\label{main+}
Assume that $\a\in[-1\slash 2,\infty)^{d}$, $\eps\in\Z$ and $m$ is of the form \eqref{molt} or \eqref{momt}. Then the Laguerre-type operators $m^{\eps,+}(\La)$, defined initially on $L^2(\R,dw_{\a}^+)$ by \eqref{eq2}, extend uniquely to bounded linear operators on $L^{p}(\R,U dw_{\a}^{+})$, $U\in A_{p}^{\a,+}$, $1<p<\infty$, and to bounded linear operators from $L^1(\R,U dw_{\a}^{+})$ to weak $L^{1}(\R,U dw_{\a}^{+})$, $U\in A_{1}^{\a,+}$.
\end{thm}

The proof of Theorem \ref{main+} will be obtained by means of the general Calder\'on-Zygmund theory. More precisely, we will show that each of the operators $m^{\eps,+}(\La)$, $\eps\in\Z$, is a Calder\'on-Zygmund operator in the sense of the space of homogeneous type $(\R,dw_\a^+,|\cdot|)$. It is well known that the classical Calder\'on-Zygmund theory works, with appropriate adjustments, when the underlying space is of homogeneous type; see for instance the comments and references in \cite[p.\,649]{NS1}.

The following result combined with the Calder\'on-Zygmund theory implies Theorem \ref{main+}, and thus also Theorem \ref{main}.
The corresponding proof splits naturally into proving Proposition \ref{assoc} and Theorem \ref{kest} below.

\begin{thm}\label{CZ}
Assume that $\a\in[-1\slash 2,\infty)^{d}$ and $\eps\in\Z$. Then the Laguerre-type operators from Theorem \ref{main+} are Calder\'on-Zygmund operators in the sense of the space of homogeneous type $(\R,dw_\a^+,|\cdot|)$.
\end{thm}

Formal computations, similar to those from \cite[Section 2]{W}, suggest that $m^{\eps,+}(\La)$, $\eps\in\Z$, where $m$ is either as in \eqref{molt} or as in \eqref{momt}, are associated with the kernels
\begin{align}\label{kera}
K^{\a,\eps}(x,y)=&\,
K_{\eta}^{\a,\eps}(x,y)=
-\int_0^\infty \partial_{t}G_t^{\a,\eps}(x,y) \, \eta(t)\,dt,
\qquad \eps\in\Z,\\\label{kerb}
\mathcal{K}^{\a,\eps}(x,y)=&\,
\mathcal{K}_{\mu}^{\a,\eps}(x,y)=
\int_0^\infty G_t^{\a,\eps}(x,y)\,d\mu(t),
\qquad \eps\in\Z,
\end{align}
respectively. The next result shows that this is indeed the case, at least in the Calder\'on-Zygmund theory sense.

\begin{pro}\label{assoc}
Let $\a\in[-1\slash 2,\infty)^{d}$ and $\eps\in\Z$.
\begin{itemize}
\item[(a)] If $m=m_{\eta}$ is a Laplace type multiplier, then $m^{\eps,+}(\La)$ is associated with the kernel $K^{\a,\eps}(x,y)$ in the sense that for any $f,g\in C^{\infty}_c(\R)$ with disjoint supports
\begin{align}\label{eq1}
\langle m^{\eps,+}(\La) f, g\rangle_{dw_\a^+}
=
\int_{\R}\int_{\R} K^{\a,\eps}(x,y) f(y)\,dw_\a^+(y) \,
\overline{g(x)}\,dw_\a^+(x).
\end{align}
\item[(b)] If $m=m_{\mu}$ is a Laplace-Stieltjes type multiplier, then $m^{\eps,+}(\La)$ is associated with the kernel $\mathcal{K}^{\a,\eps}(x,y)$ in the sense that for any $f,g\in C^{\infty}_c(\R)$ with disjoint supports
\begin{align}\label{eq3}
\langle m^{\eps,+}(\La) f, g\rangle_{dw_\a^+}
=
\int_{\R}\int_{\R} \mathcal{K}^{\a,\eps}(x,y) f(y)\,dw_\a^+(y) \,
\overline{g(x)}\,dw_\a^+(x).
\end{align}
\end{itemize}
\end{pro}

\begin {proof}
Taking into account that for each $\eps\in\Z$ the system $\{2^{d\slash 2}h_k^\a : k\in\mathcal{N}_{\eps}\}$ is an orthonormal basis in $L^2(\R,dw_\a^+)$, by \eqref{eq2} and Parseval's identity we see that
\begin{align}\label{eq4}
\langle m^{\eps,+}(\La) f, g\rangle_{dw_\a^+}
=
\sum_{n=0}^\infty m(\ll_n^\a) 
\sum_{\substack{|k|=n\\ k\in \mathcal{N}_{\eps}}}
\langle f,h_k^{\a} \rangle_{dw_{\a}^+}
\langle h_k^{\a},g \rangle_{dw_{\a}^+}.
\end{align}
To finish the proof it suffices to show that the right-hand side of \eqref{eq4} coincides with the right-hand side of \eqref{eq1} or \eqref{eq3}, according to whether $m$ is as in \eqref{molt} or \eqref{momt}, respectively. This task reduces to justifying the possibility of changing the order of integration, summation and differentiation in the relevant expressions. Then the arguments are similar as for other operators, see for instance \cite[Proposition 3.3]{NS1}. The crucial facts are the estimate (cf. \cite[(2.3)]{Sz1}), 
\begin{align*}
|h_{k}^{\a}(x)|
&\lesssim 
(|k|+1)^{c_{d,\a}}
,\qquad k\in \N^{d},\quad x\in\R,
\end{align*}
and the bounds
\begin{align*}
\int_0^\infty \big|\partial_t G_t^{\a,\eps}(x,y) \, \eta(t)\big|\,dt
\lesssim&
\frac{1}{w_{\alpha}^{+}(B(x,|y-x|))},\qquad x,y\in\R,\quad x\ne y,\\
\int_0^\infty G_t^{\a,\eps}(x,y)\,d|\mu|(t)
\lesssim&
\frac{1}{w_{\alpha}^{+}(B(x,|y-x|))},\qquad x,y\in\R,\quad x\ne y,
\end{align*}
which are verified in the proof of Theorem \ref{kest}. 
We leave further details to the reader.
\end {proof}

Let $B(x,r)$ denote the ball centered at $x$ and of radius $r$, restricted to $\R$.
\begin{thm}\label{kest}
Assume that $\a\in[-1\slash 2,\infty)^{d}$ and $\eps\in\Z$.
\begin{itemize}
\item[(a)]
The kernel $K^{\a,\eps}(x,y)$ satisfies the growth estimate
\begin{align*}
|K^{\a,\eps}(x,y)|
&\lesssim
\frac{1}{w_{\alpha}^{+}(B(x,|y-x|))},\qquad x,y\in\R,\quad x\ne y,
\end{align*}
and the gradient condition
\begin{align*}
|\nabla_{\!x,y}K^{\a,\eps}(x,y)|
&\lesssim
\frac{1}{|x-y|} \;
\frac{1}{w_{\alpha}^{+}(B(x,|y-x|))},\qquad x,y\in\R,\quad x\ne y.
\end{align*}
\item[(b)]
Analogous estimates hold for $\mathcal{K}^{\a,\eps}(x,y)$.
\end{itemize}
\end{thm}

The proof of Theorem \ref{kest} is the most technical part of the paper and is located in Section \ref{ker}.

We conclude this section with various comments and remarks related to the main result. 
First, note that our methods allow also to treat Laplace and Laplace-Stieltjes type multipliers related to the square root of $\La$. Indeed, consider the multiplier operator
\begin{align}\label{mPoisson}
m(\sqrt{\La}) f=
\sum_{n=0}^{\infty} m(\sqrt{\ll_n^\a})\sum_{|k|=n}\langle f,h_k^{\alpha}
  \rangle_{dw_{\alpha}}h_k^{\alpha}, 
\qquad f\in L^2(\RR,dw_{\alpha}),
\end{align} 
and for $\eps\in\Z$ the Poisson-Laguerre-type operators
\begin{align}\label{PoiLag}
m^{\eps,+}(\sqrt{\mathcal{L}_{\a}}) f
=
\sum_{n=0}^{\infty} m(\sqrt{\ll_n^\a})
\sum_{\substack{|k|=n\\ k\in \mathcal{N}_{\eps}}}
\langle f,h_k^{\alpha}
  \rangle_{dw_{\alpha}^+}h_k^{\alpha},
\qquad f\in L^2(\R,dw_{\alpha}^{+}),
\end{align}
where $m$ is as in Definition \ref{def} with $\ll_0^{\a}$ replaced by $\sqrt{\ll_0^{\a}}$. Then the operators $m(\sqrt{\La})$ and $m^{\eps,+}(\sqrt{\mathcal{L}_{\a}})$ are well defined and bounded on $L^2(\RR,dw_{\a})$ and $L^2(\R,dw_{\a}^{+})$, respectively. The integral kernels associated with $m^{\eps,+}(\sqrt{\mathcal{L}_{\a}})$ have analogous forms to \eqref{kera} and \eqref{kerb}, where $G_t^{\a,\eps}(x,y)$ should be replaced by the subordinated kernel
\begin{align*}
P_t^{\a,\eps} (x,y)=\int_0^\infty \,G_{t^{2}\slash (4u)}^{\a,\eps} (x,y)\,
\frac{e^{-u}\,du}{\sqrt{\pi u}}.
\end{align*}
Proceeding similarly as in the case of $m(\La)$, with only slightly more effort we obtain the following result.

\begin{thm}\label{Poisson1}
Assume that $\a\in[-1\slash 2,\infty)^{d}$ and $\eps\in\Z$. Then the Poisson-Laguerre-type operators $m^{\eps,+}(\sqrt{\La})$, $\eps\in\Z$, defined by \eqref{PoiLag}, are Calder\'on-Zygmund operators in the sense of the space of homogeneous type $(\R,dw_\a^+,|\cdot|)$. Consequently, each of these operators, defined initially on $L^2(\R,dw_{\a}^+)$, extends uniquely to a bounded linear operator on $L^{p}(\R,U dw_{\a}^{+})$, $U\in A_{p}^{\a,+}$, $1<p<\infty$, and to a bounded linear operator from $L^1(\R,U dw_{\a}^{+})$ to weak $L^{1}(\R,U dw_{\a}^{+})$, $U\in A_{1}^{\a,+}$.
\end{thm}

\begin{cor}\label{mainP}
Assume that $\a\in[-1\slash 2,\infty)^{d}$ and $W$ is a weight on $\RR$ invariant under the reflections $\sigma_1, \ldots,\sigma_d$. Then the multiplier operator $m(\sqrt{\La})$, defined initially on $L^2(\RR,dw_\a)$ by \eqref{mPoisson}, extends uniquely to a bounded linear operator on $L^{p}(\RR,W dw_{\a})$, $W^+\in A_{p}^{\a,+}$, $1<p<\infty$, and to a bounded linear operator from $L^1(\RR,W dw_{\a})$ to weak $L^1(\RR,W dw_{\a})$, $W^+\in A_{1}^{\a,+}$.
\end{cor} 

Next, we comment on the relation between the auxiliary operators $m^{\eps,+}(\La)$, $\eps\in\Z$, and the Laguerre setting from \cite{NS1}. Note that for the particular $\eps_{0}=(0,\ldots,0)$ the Laguerre-type operator $m^{\eps_{0},+}(\La)$ coincides, up to the factor $2^{-d}$, with the multiplier operator $m(\La^{\ell})$ related to the Laguerre Laplacian $\La^{\ell}$ considered in \cite{NS1} and \cite{Sz}. More precisely,
\begin{align}\label{eq7}
m(\La^{\ell}) f=
\sum_{n=0}^{\infty} m(\lambda_{2n}^\a)\sum_{|k|=n}\langle f,\ell_k^{\alpha}
  \rangle_{d\mu_{\alpha}}\ell_k^{\alpha}, 
\qquad f\in L^2(\R,d\mu_{\alpha}),
\end{align}
where $\ell_k^\a$ are the Laguerre functions of convolution type and $\mu_\a\equiv w_\a^+$. Therefore the results of this section deliver also analogous results in the Laguerre setting of convolution type; see \cite[Section 2]{Sz1} for further explanations concerning the connection between the Dunkl and Laguerre settings.

\begin{thm}\label{Lag}
Assume that $\a\in[-1\slash 2,\infty)^{d}$ and $m$ is as in \eqref{molt} or \eqref{momt}. Then the Laguerre multiplier operator $m(\La^{\ell})$, defined initially on $L^2(\R,d\mu_\a)$ by \eqref{eq7}, is a Calder\'on-Zygmund operator in the sense of the space of homogeneous type $(\R,d\mu_\a,|\cdot|)$. Consequently, $m(\La^{\ell})$ extends uniquely to a bounded linear operator on $L^{p}(\R,U d\mu_{\a})$, $U\in A_{p}^{\a,+}$, $1<p<\infty$, and to a bounded linear operator from $L^1(\R,U d\mu_{\a})$ to weak $L^1(\R,U d\mu_{\a})$, $U\in A_{1}^{\a,+}$.
\end{thm}

A similar result holds for Laplace and Laplace-Stieltjes type multipliers related to the square root of $\La^{\ell}$, see Theorem \ref{Poisson1} and Corollary \ref{mainP} above.
\section{Kernel estimates}\label{ker}
\setcounter{equation}{0}

This section delivers proofs of the relevant kernel estimates. We use the technique developed by Nowak and Stempak \cite{NS1,NS3,NS2}, which is based on Schl\"afli's integral representation for the modified Bessel function $I_{\nu}$ involved in the Dunkl heat kernel. This method was refined by the author in \cite{Sz} and \cite{Sz1} to obtain standard estimates for various kernels related to $\La^{\ell}$ and $\La$. Below, we will frequently invoke estimates obtained in the latter paper. Recall that we always assume that $\a\in[-1\slash 2, \infty)^d$.

Given $\eps\in\Z$, the $\eps$-component of the Dunkl heat kernel is given by, see \cite[Section 5]{NS3},
\begin{align}\nonumber
&G_{t}^{\a,\eps}(x,y)\\\label{hk}
&\quad=
\frac{1}{2^{d}}\Big(\frac{1-\zeta^{2}}{2\zeta}\Big)^{d+|\alpha|+|\eps|}(xy)^{\eps}
\int_{[-1,1]^{d}}\exp\Big(-{\frac{1}{4\zeta}}q_{+}(x,y,s)-\frac{\zeta}{4}q_{-}(x,y,s)\Big)\,\Pi_{\a+\eps}(ds),
\end{align}
where $(xy)^{\eps}=(x_{1}y_{1})^{\eps_{1}}\cdot\ldots\cdot (x_{d}y_{d})^{\eps_{d}}$,
\begin{align*}
q_{\pm}(x,y,s)=
|x|^2+|y|^2\pm 2\sum_{i=1}^{d}x_{i}y_{i}s_{i}
\end{align*}
and $t>0$ and $\z\in(0,1)$ are related by $\zeta=\tanh t$; equivalently
\begin{equation}\label{zz}
t=t(\z)=\frac{1}{2}\log\frac{1+\z}{1-\z}.
\end{equation}
The measure $\Pi_{\beta}$ appearing in \eqref{hk} is a product of one-dimensional measures, $\Pi_{\beta}=\bigotimes_{i=1}^{d}\Pi_{\beta_{i}}$, where $\Pi_{\beta_{i}}$ is given by the density
$$
\Pi_{\beta_{i}}(ds_{i}) = \frac{(1-s_{i}^2)^{\beta_{i}-1\slash 2}ds_{i}}{\sqrt{\pi} 2^{\beta_{i}}\Gamma{(\beta_{i}+1\slash 2)}},
    \qquad \beta_{i} > -1\slash 2,
$$
and in the limiting case $\Pi_{-1\slash 2}=\big( \delta_{-1} + \delta_1 \big)\slash \sqrt{2\pi}$, with $\delta_{-1}$ and $\delta_{1}$ denoting the point masses at $-1$ and $1$, respectively.

To estimate the kernels defined via $G_{t}^{\a,\eps}(x,y)$ we will need several auxiliary results. In particular, the following modification of \cite[Lemma 1.1]{StTo} will be useful.

\begin{lem}\label{mod}
Given $a>1$, we have
\begin{equation*}
\int_0^1(1-\z^2)^{-1\slash 2}\z^{-a}\exp(-T\zeta^{-1})\,d\z\lesssim T^{-a+1},
\qquad T>0.
\end{equation*}
\end{lem}

\begin {proof}
We split the region of integration onto $(0,1\slash 2)$ and $(1\slash 2,1)$. For $\z \in (0,1\slash 2)$ we have $(1-\z^2)^{-1\slash 2}\simeq 1$, so the bound for the integral over $(0,1\slash 2)$ is a straightforward consequence of \cite[Lemma 1.1]{StTo}. It remains to estimate the integral over $(1\slash 2,1)$. Since $\z \simeq 1$ in this case and $\sup_{u\geq 0}u^{a-1}e^{-u}<\infty$, we see that
\begin{align*}
\z^{-a}\exp(-T\zeta^{-1})
\lesssim
T^{-a+1}(T\z^{-1})^{a-1}\exp(-T\zeta^{-1})
\lesssim
 T^{-a+1},\qquad \z\in(1\slash 2,1),\quad T>0.
\end{align*}
Now the conclusion follows because $\int_{1\slash 2}^{1}(1-\z^2)^{-1\slash 2}\,d\z<\infty$.
\end {proof}

The lemma below establishes an important connection between the estimates emerging from the representation \eqref{hk} and the standard estimates related to the space $(\R,dw_\a^+,|\cdot|)$.

\begin{lem}$($\cite[Lemma 5.3]{NS3}, \cite[Lemma 4]{NS2}$)$ \label{lem4}
Assume that $\alpha \in [-1\slash 2, \infty)^d$ and let $\delta,\kappa \in [0,\infty)^d$ be fixed.
Then for $x,y\in\R$, $x \neq y$,
\begin{equation*}
(x+y)^{2\delta} \int_{[-1,1]^d}
	 \big(q_{+}(x,y,s)\big)^{-d - |\alpha| -|\delta|}\,\Pi_{\alpha+\delta+\kappa}(ds)
\lesssim \frac{1}{w_{\alpha}^+(B(x,|y-x|))}
\end{equation*}
and
\begin{equation*}
 (x+y)^{2\delta}\int_{[-1,1]^d} 
    \big(q_{+}(x,y,s)\big)^{-d - |\alpha| -|\delta| - 1\slash 2}\,\Pi_{\alpha+\delta+\kappa}(ds)
\lesssim \frac{1}{|x-y|} \;
 \frac{1}{w_{\alpha}^+(B(x,|y-x|))}.
\end{equation*}
\end{lem}

To state the next lemma and to perform the relevant kernel estimates we will use the same abbreviation as in \cite{Sz1},
$$
\e\b=\exp\Big(-{\frac{1}{4\zeta}}q_{+}(x,y,s)-\frac{\zeta}{4}q_{-}(x,y,s)\Big).
$$
Also, we will often neglect the set of integration $[-1,1]^{d}$ in integrals against $\Pi_{\alpha}$ and will frequently write shortly $q_{+}$ and $q_{-}$ omitting the arguments.

\begin{lem}\label{lem1}
Assume that $\a \in [-1\slash 2,\infty)^{d}$ and $\eps, \xi, \rho \in \Z$ are fixed and such that $\xi \leq \eps$, $\rho \leq \eps$. Given $C>0$ and $u \geq 1$, define the function $p_{u}$ acting on $\R \times \R \times (0,1)$ by
\begin{align*}
p_{u}(x,y,\z)=
\sqrt{1-\z^2}\,\z^
{-d-|\a|-|\eps|+|\xi|\slash 2+|\rho|\slash 2-u\slash 2-1\slash 2}
\,x^{\eps-\xi}y^{\eps-\rho}\int_{[-1,1]^{d}}\big(\ee\b\big)^{C}\,\Pi_{\a+\eps}(ds).
\end{align*}
Then $p_u$ satisfies the integral estimate
\begin{align*}
\|p_{u}(x,y,\z(t))\|_{L^1(dt)}
&\lesssim
\frac{1}{|x-y|^{u-1}} \;
\frac{1}{w^{+}_{\alpha}(B(x,|y-x|))},\qquad x\ne y,
\end{align*}
where $t$ and $\z$ are related as in $\eqref{zz}$.
\end{lem}

\begin {proof}
Changing the variable as in \eqref{zz} and then using sequently the Fubini-Tonelli theorem, Lemma \ref{mod} (specified to $a=d+|\a|+|\eps|-|\xi|\slash 2-|\rho|\slash 2+u\slash 2+1\slash 2$) and the inequality $|x-y|^2\leq q_{+}$, we obtain
\begin{align*}
\|p_{u}&(x,y,\z(t))\|_{L^1(dt)}\\
=&\,
x^{\eps-\xi}y^{\eps-\rho}
\int\int_0^1(1-\z^2)^{-1\slash 2}
\z^{-d-|\a|-|\eps|+|\xi|\slash 2+|\rho|\slash 2-u\slash 2-1\slash 2}
\big(\e\b\big)^{C}\,d\z\,\Pi_{\a+\eps}(ds)\\
\lesssim&\,
x^{\eps-\xi}y^{\eps-\rho}
\int(q_+)^{-d-|\a|-|\eps|+|\xi|\slash 2+|\rho|\slash 2-u\slash 2+1\slash 2}
\,\Pi_{\a+\eps}(ds)\\
\leq&\,
(x+y)^{2\eps-\xi-\rho}
\frac{1}{|x-y|^{u-1}}
\int(q_+)^{-d-|\a|-|\eps-\xi\slash 2-\rho\slash 2|}\,\Pi_{\a+\eps}(ds).
\end{align*}
This, in view of Lemma \ref{lem4} (taken with $\delta=\eps-\xi\slash 2-\rho\slash 2$ and $\kappa=\xi\slash 2+\rho\slash 2$), gives the desired conclusion.
\end {proof}

\begin{lem}$($\cite[Lemma 4.2]{Sz}$)$\label{oq}
Given $b\geq 0$ and $c>0$, we have
$$
\big(q_{\pm}(x,y,s)\big)^{b}\exp\big(-cAq_{\pm}(x,y,s)\big)\lesssim A^{-b}\exp\Big(\frac{-cA}{2}q_{\pm}(x,y,s)\Big),\qquad A>0,
$$
uniformly in $q_{\pm}$.
\end{lem}


\begin {proof}[Proof of Theorem \ref{kest} (a)]
The growth condition is a consequence of the estimate
\begin{align*}
\big |\partial_{t}G_{t}^{\a,\eps}(x,y) \big|
\lesssim
\sqrt{1-\z^2}\,\z^
{-d-|\a|-|\eps|-1}\,(xy)^{\eps}
\int\big(\e\b\big)^{1\slash 2}\,\Pi_{\a+\eps}(ds),
\end{align*}
which is stated explicitly in \cite[(4.6)]{Sz1}, the fact that $\eta$ is bounded and Lemma \ref{lem1} (applied with $u=1$, $\xi=\rho=0$).

To prove the gradient condition, by symmetry reasons, it suffices to show that
\begin{align*}
\big| \partial_{x_j}K^{\a,\eps}(x,y) \big|
\lesssim
\frac{1}{|x-y|} \;
\frac{1}{w_{\a}^+(B(x,|y-x|))},\qquad x\ne y,\quad j=1,\ldots,d.
\end{align*}
Differentiating \eqref{kera} in $x_j$ (passing with $\partial_{x_{j}}$ under the integral sign can be easily justified, see \cite[Section 4.1]{Sz}) we get
\begin{align*}
\partial_{x_j}K^{\a,\eps}(x,y)
=
-\int_0^\infty \partial_{x_j}\partial_t G_t^{\a,\eps}(x,y) \, \eta(t)\,dt,\qquad x\ne y.
\end{align*}
Applying now the first inequality in \cite[(4.7)]{Sz1},
\begin{align*}
\big| \partial_{x_j}\partial_{t}G_{t}^{\a,\eps}(x,y) \big|
\lesssim&
\sqrt{1-\z^2}\,\z^{-d-|\a|-|\eps|-3\slash 2}\,(xy)^{\eps}
\int\big(\e\b\big)^{1\slash 4}\,\Pi_{\a+\eps}(ds)\\
&+
\chi_{\{\eps_{j}=1\}}
\sqrt{1-\z^2}\,\z^{-d-|\a|-|\eps|-1}\,x^{\eps-e_j}y^{\eps}
\int\big(\e\b\big)^{1\slash 2}\,\Pi_{\a+\eps}(ds),
\end{align*}
the fact that $\eta$ is bounded and Lemma \ref{lem1} twice (specified to $u=2$ and either $\xi=\rho=0$ or $\xi=e_j$, $\rho=0$) leads directly to the required bound.

The proof of (a) in Theorem \ref{kest} is finished.
\end {proof}


\begin {proof}[Proof of Theorem \ref{kest} (b)]
We first verify the growth condition. Using Lemma \ref{oq} (taken with $b=d+|\a|+|\eps|$, $c=1\slash 4$, $A=\z^{-1}$) we get
\begin{align*}
e^{t(2d+2|\a|)}G_t^{\a,\eps}(x,y)
\leq&
\Big(\frac{1+\z}{1-\z}\Big)^{d+|\a|}
\Big(\frac{1-\z^2}{2\z}\Big)^{d+|\a|+|\eps|}
(xy)^{\eps}
\int\e\b\,\Pi_{\a+\eps}(ds)\\
\lesssim&\,
\z^{-d-|\a|-|\eps|}\,(xy)^{\eps}\int\e\b\,\Pi_{\a+\eps}(ds)\\
\lesssim&
(x+y)^{2\eps}
\int(q_+)^{-d-|\a|-|\eps|}\,\Pi_{\a+\eps}(ds).
\end{align*}
Now the conclusion follows with the aid of Lemma \ref{lem4} (specified to $\delta=\eps$ and $\kappa=0$) and the assumption \eqref{meas} concerning the measure $\mu$.

Next, our task is to show the gradient condition. By symmetry reasons, it suffices to prove that
\begin{align*}
\big| \partial_{x_j}\mathcal{K}^{\a,\eps}(x,y) \big|
\lesssim
\frac{1}{|x-y|} \;
\frac{1}{w_{\a}^+(B(x,|y-x|))},\qquad x\ne y,\quad j=1,\ldots,d.
\end{align*}
Differentiating \eqref{kerb} in $x_j$ (exchanging $\partial_{x_{j}}$ with the integral sign is legitimate, which can be justified by a slightly modified version of \cite[Lemma 4.7]{Sz1}, see for instance \cite[Section 4.3]{Sz}) we get
\begin{align*}
\partial_{x_j}\mathcal{K}^{\a,\eps}(x,y)
=
\int_0^\infty \partial_{x_j} G_t^{\a,\eps}(x,y)\,d\mu(t),\qquad x\ne y.
\end{align*}
Using the estimate
\begin{align*}
\big| \partial_{x_j}G_t^{\a,\eps}(x,y) \big|
\lesssim&
\Big(\frac{1-\z^2}{2\z}\Big)^{d+|\a|+|\eps|}
\z^{-1\slash 2}\,(xy)^{\eps}
\int\big(\e\b\big)^{1\slash 2}\,\Pi_{\a+\eps}(ds)\\
&+
\chi_{\{\eps_{j}=1\}}
\Big(\frac{1-\z^2}{2\z}\Big)^{d+|\a|+|\eps|}
\,x^{\eps-e_j}y^{\eps}
\int\e\b\,\Pi_{\a+\eps}(ds),
\end{align*}
which is implicitly contained in \cite[Section 4.2]{Sz1}, and then proceeding in a similar way as in the case of the growth condition, applying this time Lemma \ref{oq} twice (with $b=d+|\a|+|\eps|+1\slash 2$, $c=1\slash 8$, $A=\z^{-1}$ and $b=d+|\a|+|\eps|$, $c=1\slash 4$, $A=\z^{-1}$), we obtain
\begin{align*}
e^{t(2d+2|\a|)} \big|\partial_{x_{j}}G_t^{\a,\eps}(x,y)\big|
\lesssim&
(x+y)^{2\eps}
\int(q_+)^{-d-|\a|-|\eps|-1\slash 2}\,\Pi_{\a+\eps}(ds)\\
&+
\chi_{\{\eps_{j}=1\}}
(x+y)^{2\eps-e_j}
\int(q_+)^{-d-|\a|-|\eps-e_{j}\slash 2|-1\slash 2}\,\Pi_{\a+\eps}(ds).
\end{align*}
Taking into account \eqref{meas}, this estimate together with Lemma \ref{lem4} (specified to $\delta=\eps$, $\kappa=0$ and $\delta=\eps-e_{j}\slash 2$, $\kappa=e_{j}\slash 2$) leads directly to the gradient condition for $\mathcal{K}^{\a,\eps}(x,y)$. This completes proving (b) in Theorem \ref{kest}.
\end {proof}


\end{document}